\documentclass[11pt]{article}
\usepackage{amsthm,amssymb,amsmath,graphicx,cite}

\newcommand{\ip}[2]{\left\langle #1 , #2 \right\rangle}    
\newcommand{\R}{{\mathbb R}}
\renewcommand{\S}{{\mathbb S}}

\newcommand{\PI}{\text{\sf PI}}
\newcommand{\Gr}{\text{\sf Gr}}

\newcommand{\sym}{{\mathrm{Sym}}}

\newcommand{\vertiii}[1]{{\vert\kern-0.25ex\vert\kern-0.25ex\vert #1 
    \vert\kern-0.25ex\vert\kern-0.25ex\vert}}

\DeclareMathOperator{\dist}{dist_G}
\DeclareMathOperator{\odist}{\overline{dist_G}}
\DeclareMathOperator{\Rdist}{dist_R}

\DeclareMathOperator{\lspan}{span}
\DeclareMathOperator{\interior}{int}
\DeclareMathOperator{\Image}{Im}

\newtheorem{lemma}{Lemma}
\newtheorem{theorem}{Theorem}
\newtheorem{proposition}{Proposition}

\newtheorem{corollary}{Corollary}

\newtheorem{example}{Example}
\def\transp{^{\text{\sf T}}}
\newcommand{\matr}[1]{\begin{bmatrix} #1 \end{bmatrix}}    

\newcommand{\tr}{\mbox{\rm trace}}

\title{On the Grassmann condition number}
\author{Javier Pe\~na\thanks{Tepper School of Business,
Carnegie Mellon University, USA, {\tt jfp@andrew.cmu.edu}} \and Vera Roshchina\thanks{School of Science,
RMIT University, AUSTRALIA, {\tt vera.roshchina@rmit.edu.au}}}

\begin{document}

\maketitle

\begin{abstract}
We give new insight into the {\em Grassmann condition} of the conic  feasibility problem
\begin{equation}\label{eq.the.system}
\text{ find }  x \in L \cap K \setminus\{0\}.
\end{equation}
Here $K\subseteq V$ is a regular convex cone and $L\subseteq V$ is a linear subspace of the finite dimensional Euclidean vector space $V$.  The Grassmann condition of \eqref{eq.the.system} is the reciprocal of the {\em distance} from $L$ to the set of {\em ill-posed} instances in the Grassmann manifold where $L$ lives.

We consider  a very general distance in the Grassmann manifold defined by two possibly different norms in $V$.  We establish the equivalence between the Grassmann distance to ill-posedness of \eqref{eq.the.system} and a natural measure of the {\em least violated} trial solution to the alternative to \eqref{eq.the.system}. We also show a tight relationship between the Grassmann  and Renegar's condition measures, and between the Grassmann measure and a {\em symmetry} measure of  \eqref{eq.the.system}.

Our approach can be readily specialized to a canonical norm in $V$ induced by $K$, a prime example being the one-norm for the non-negative orthant.  For this special case we show that the Grassmann distance ill-posedness of \eqref{eq.the.system} is equivalent to a measure of the {\em most interior} solution to \eqref{eq.the.system}.

\end{abstract}

\newpage

\section{Introduction}
Assume $V$ is a finite dimensional Euclidean vector space and $K\subseteq V$ is a closed convex cone.  Given a linear subspace $L\subseteq V$, consider the feasibility problem
\begin{equation}\label{primal}
\text{ find} \; x\in L\cap K \setminus\{0\}.
\end{equation}
Feasibility problems of this form are pervasive in optimization.  The  constraints of linear, semidefinite, and more general conic programming are written in the form~\eqref{primal}.  The fundamental signal recovery property in compressive sensing can be stated precisely as the infeasibility of \eqref{primal} for suitable $K$ and $L$ as it is nicely explained in~\cite{AmelLMT14}.

We develop novel insights into  the {\em Grassmann condition} of  \eqref{primal}, which is defined as follows.  Let $\Gr(V,m)$ denote the Grassmann manifold of  linear subspaces $L\subseteq V$ of dimension $m<\dim(V)$.  Given $L\in \Gr(V,m)$ consider the following measure of well-posedness of the problem \eqref{primal}: How far, in some suitable distance in $\Gr(V,m)$, can $L$ be perturbed without changing the feasibility status of \eqref{primal}.    The size of the largest such perturbation of $L$ is the {\em Grassmann distance to ill-posedness}.
The reciprocal of this distance to ill-posedness is the  {\em Grassmann condition} of $L$.

The Grassmann condition is a relatively recent development in the literature on conditioning for optimization initiated by Renegar~\cite{Rene95a,Rene95} and further studied by a number of authors~\cite{BellF09a,CheuC01,CheuCP03,DontLR03,EpelF00,EpelF02,Freu04,FreuV99a,FreuV99b,Lewi99,Pena00}.   In contrast to Renegar's condition, which inherently depends on the {\em data} of some linear operator defining $L$, the Grassmann condition is data independent as it is entirely constructed in terms of objects in the Grassmann manifold.  The initial developments on the Grassmann condition can be traced to the articles~\cite{AmelB12,BellF09a,CheuCP10}.  This Grassmann condition  machinery has in turn been instrumental in further developments concerning the average behavior of condition numbers~\cite{AmelB15}.    Our work is inspired by the main ideas in~\cite{AmelB12} and~\cite{CheuCP10}.  In spite of their conceptual overlap, the two articles~\cite{AmelB12,CheuCP10} were completed independently of each other.  On the one hand, Amelunxen and Burgisser~\cite{AmelB12}, extending ideas from Belloni and Freund~\cite{BellF09a}, propose a Grassmann condition by working with a canonical projection-based distance in $\Gr(V,m)$.  They establish an interesting connection between the Grassmann and Renegar's condition. They also give some characterizations of the Grassmann condition in terms of angles.
 On the other hand, Cheung, Cucker and Pe\~na~\cite{CheuCP10} introduce essentially the same concept of Grassmann condition albeit in a different format.  They use a more general type of distance in the Grassmann manifold and a more nuanced type of ill-posedness but their development is restricted to the special case  $V=\R^n,\; K = \R^n_+$.  Cheung et al.~\cite{CheuCP10} establish several equivalences and bounds between the Grassmann condition and a variety of other condition measures for polyhedral feasibility problems.

This article combines and extends the main ideas and results from both \cite{AmelB12} and \cite{CheuCP10}.  Our approach applies to general regular cones as in~\cite{AmelB12} but we use a very general distance function in $\Gr(V,m)$ in the same spirit of~\cite{CheuCP10}.  Section~\ref{sec.Grassmann} below details the construction of the relevant distance function.  Our first main result (Theorem~\ref{thm:Grassmann.nu}) shows that when \eqref{primal} is feasible, its Grassmann distance to ill-posedness matches a  natural measure of the {\em least violated trial solution} to the alternative system to~\eqref{primal} namely
\begin{equation}\label{dual}
u\in L^\perp\cap K^*  \setminus\{0\}.
\end{equation}
Here $L^\perp$ and $K^*$ are  the orthogonal complement of $L$ and the dual cone of $K$ respectively. We also give a counterpart of this result for the case when \eqref{primal} is infeasible, that is, when \eqref{dual} is feasible.  This first result extends \cite[Theorem 1]{CheuCP10} as well as \cite[Proposition 1.6]{AmelB12}.  

Our second main result (Theorem~\ref{thm.renegar}) establishes a tight connection between the Grassmann condition and Renegar's condition.  In the particular case when the space $L$ is defined via a linear isometry, Theorem~\ref{thm.renegar} shows that the Grassmann and Renegar's condition measures are exactly the same.  Otherwise Renegar's condition is sandwiched between the Grassmann condition and the Grassmann condition times the condition number of the linear operator defining $L$.  The statement of Theorem~\ref{thm.renegar} is nearly identical to \cite[Theorem 1.4]{AmelB12}.  Our approach reveals  that the relationship between Grassmann and Renegar's condition numbers indeed holds in much wider generality.

Our general approach can be seamlessly specialized to a canonical norm {\em induced} by $K$.  Prime examples of this induced norm are the one-norm for the non-negative orthant and the nuclear norm for the cone of positive semidefinite matrices.   
When one of the norms in $V$ is the norm induced by  $K$, the Grassmann distance to ill-posedness of \eqref{primal} can be characterized in terms of the {\em most interior} solution to~\eqref{primal}.  We also introduce a related family of {\em eigenvalue mappings} induced by $K$ that allows us to relate the Grassmann condition with another popular condition measure, namely the {\em symmetry measure} which has been used by several authors~\cite{BellF08,BellF09b,BertGS11,EpelF02}.  The exact connection between the symmetry measure and other condition measures has only been explored in the special  case $V=\R^n, \; K=\R^n_+$ by Epelman and Freund~\cite{EpelF02}.   We show that some of the results concerning the symmetry measure established in~\cite{EpelF02} hold in more generality and are nicely related to the Grassmann condition measure (Theorem~\ref{thm.symmetry}).

It should be noted that we pay a small price for the of generality of our approach.  Our formal results are stated in terms of two related but different distances in the Grassmann manifold, as we detail in Section~\ref{sec.Grassmann} below.  This is necessary due to a certain asymmetry in these distance functions.  The statements in Theorem~\ref{thm:Grassmann.nu} and Theorem~\ref{thm.renegar} reflect this asymmetry: the claims concerning the distance to ill-posedness of \eqref{primal} are  slightly different in the feasible and infeasible cases. In the special case of Euclidean norms, this kind of asymmetry disappears and the relevant quantities from  Theorem~\ref{thm:Grassmann.nu} can be equivalently phrased in terms of principal angles (Proposition~\ref{prop.euclidean}).   The particular Euclidean norm case also subsumes some of the main results from \cite{AmelB12} within the context of our work.

We mention in passing the work by Lewis~\cite{Lewi99} and Lewis and co-authors~\cite{BorwL06,DontLR03} that was also inspired by Renegar's seminal work on condition numbers~\cite{Rene95a,Rene95}.  The article~\cite{Lewi99} shows that Renegar's distance to ill-posedness can be seen as a special case of the {\em distance to non-surjectivity} of sublinear set-valued mappings.  The textbook~\cite{BorwL06} also includes a related discussion~\cite[Section 5.4]{BorwL06}.  The  article~\cite{Pena04} subsequently showed that the generalization in \cite{Lewi99} is a special case of a generic and natural correspondence between conic systems and sublinear set-valued mappings.  The article~\cite{DontLR03} extends the ideas in~\cite{Lewi99} to the more general concept of {\em radius of metric regularity} of set-valued mappings.  Like Renegar's approach, the underlying setting in~\cite{BorwL06,DontLR03,Lewi99} is inherently data-dependent in contrast to the data-independent approach of this article.  Although a connection between our work and~\cite{BorwL06,DontLR03,Lewi99} might exist, the sharp difference in data dependence between the two approaches makes that potential connection far from evident.

The main sections of the paper are organized as follows.  Section~\ref{sec.Grassmann} presents our main technical developments.  We construct two distances $\dist$ and $\odist$ in $\Gr(V,m)$ determined by two arbitrary norms in $V$.  We show (Theorem~\ref{thm:Grassmann.nu}) the equivalence between the distance to ill-posedness of \eqref{primal} and a measure of the trial solution that least violates its alternative~\eqref{dual}.  We also establish a tight relationship between the Grassmann condition and Renegar's condition (Theorem~\ref{thm.renegar}).  Section~\ref{sec.norms} specializes our development to some particular choices of norms in $V$, namely the Euclidean norm and a norm induced by $K$.  For Euclidean norms, we show that our main construction and results can be stated in terms of  {\em principal angles.}  For the induced norm we show that the distance to ill-posedness coincides with a natural measure of the {\em most centered} solution to \eqref{primal}.  Section~\ref{sec.symmetry.meas} shows an additional connection between the Grassmann distance and the {\em symmetry} condition studied by Belloni, Freund, and others~\cite{BellF08,BellF09b,BertGS11,EpelF02}.  Finally Section~\ref{sec.symm.poly} specializes our developments to two particularly important cases.  The first one is the case when $V$ is an Euclidean Jordan algebra and $K$ is the cone of squares in $V$.  The rich algebraic structure of this case in turn yields an interesting structure on the quantities characterizing the Grassmann condition and the symmetry condition.  The second special case is $V=\R^n, K=\R^n_+$.  The peculiar structure of this case enables us to refine our main developments by relying on the classical Goldman-Tucker partition theorem.

\section{The Grassmann condition measure}
\label{sec.Grassmann}

Throughout the entire paper we assume $V$ is a finite dimensional Euclidean vector space with inner product $\ip{\cdot}{\cdot}$ and $K\subseteq V$ is a regular closed convex cone.  We also assume that $\|\cdot\|, \; \vertiii{\cdot}$ are norms in $V$, not necessarily the same, and neither of them necessarily Euclidean.  We let $\|\cdot\|^*, \vertiii{\cdot}^*$ denote the dual norms of $\|\cdot\|$ and $\vertiii{\cdot}$ respectively.  That is, for $u \in V$
$$\|u\|^* := \max_{\|x\|=1}\ip{u}{x}, \; \text{ and } \;
\vertiii{u}^* := \max_{\vertiii{x}=1}\ip{u}{x}.$$
Notice that by construction, the following  {\em H\"older's inequality} holds for all $u,x\in V$
\begin{equation}\label{eq.holder}
|\ip{u}{x}| \le \|u\|^* \cdot \|x\|, \quad |\ip{u}{x}| \le \vertiii{u}^* \cdot \vertiii{x}.
\end{equation}
We let $K^*$ denote the {\em dual cone} of $K$, that is,
\[
K^*:=\{u\in V: \ip{u}{x} \ge 0  \; \forall x\in K\}.
\]
Given a linear subspace $L \subseteq V$, let $L^\perp$ denote the {\em orthogonal complement} of $L$ in $V$, that is,
\[
L^\perp:= \{u\in V: \ip{u}{x} = 0  \; \forall x\in L\}.
\]
Our central objects of attention are the two feasibility problems \eqref{primal} and \eqref{dual} defined by a linear subspace $L\subseteq V$.   Observe that \eqref{primal} and \eqref{dual} are {\em alternative systems:} one of them has a strictly feasible solution if and only if the other one is infeasible.

For ease of notation, we will use the following convention throughout the paper.  The symbols  $v,x,z$ will denote ``primal vectors'', that is, elements of $L$ and $K$.  On the other hand, the symbols  $u,w,y$ will denote ``dual vectors'', that is, elements of $L^\perp$ and $K^*$.

Define the distance functions $\dist,\odist: \Gr(V,m) \times \Gr(V,m)\rightarrow \R$ as follows
\[
\dist(L_1,L_2):= \max_{x\in L_1\atop x\ne 0} \min_{v\in L_2} \frac{\vertiii{x-v}}{\|x\|}, \; \odist(L_1,L_2):= \max_{x\in L_1\atop x\ne 0} \inf_{v\in L_2\atop v\ne 0} \frac{\vertiii{x-v}}{\|v\|}.
\]
The infimum in the definition of $\odist(\cdot,\cdot)$ is essential.  For instance, if $\|\cdot\|=\vertiii{\cdot}=\|\cdot\|_2$ and $L_1\subseteq L_2^\perp$ then 
$\displaystyle\inf_{v\in L_2,v\ne 0} \frac{\|x-v\|_2}{\|v\|_2}=1$ for all $x\in L_1,\, x\ne 0$ but this infimum is not attained. 

The  distances $\dist(\cdot,\cdot)$ and $\odist(\cdot,\cdot)$ are not necessarily symmetric: In general, it may be the case that $\dist(L_1,L_2)\neq \dist(L_2,L_1)$ and $\odist(L_1,L_2)\neq \odist(L_2,L_1)$.  For example, let $L_1,L_2\in\Gr(\R^2,1)$ be 
\[
L_1 = \lspan\left\{\matr{1\\0}\right\}, \; L_2 = \lspan\left\{\matr{1\\1}\right\}.
\]
For these subspaces and for $\|\cdot\|=\vertiii{\cdot}=\|\cdot\|_1$ it is easy to see that  
\[
\dist(L_1,L_2)=1\ne \frac{1}{2} =\dist(L_2,L_1)\]
and \[ \odist(L_1,L_2)=\frac{1}{2} \ne 1 = \odist(L_2,L_1).\]
The distances  $\dist(\cdot,\cdot)$ and $\odist(\cdot,\cdot)$ are bounded above as follows
\begin{equation}\label{eq.upperbounds.dist}
\dist(L_1,L_2) \le \max_{x\in L_1\atop x\ne 0}\frac{\vertiii{x}}{\|x\|}, \; \odist(L_1,L_2) \le \min_{v\in L_2\atop v\ne 0}\frac{\vertiii{v}}{\|v\|}.
\end{equation}
In the special case when $\|\cdot \| = \vertiii{\cdot} = \|\cdot\|_2$ the distances $\dist(\cdot,\cdot)$ and $\odist(\cdot,\cdot)$ coincide and are both symmetric (see Proposition~\ref{prop.euclidean} below).

Define the set $\Sigma_m \subseteq \Gr(V,m)$ of $m$-dimensional {\em ill-posed subspaces} with respect to the cone $K$ as follows
$$
\Sigma_m := \{L\in \Gr(V,m)\,|\, L\cap K \neq \{0\} \text{ and } L^\perp\cap K^*  \ne \{0\}\}. 
$$
Observe that $L\in \Gr(V,m) \setminus\Sigma_m$ if and only if either $L\cap \interior(K) \ne \emptyset$ or $L^\perp\cap \interior(K^*) \ne \emptyset.$ 
Define the {\em Grassmann distances to ill-posedness} of $L\in \Gr(V,m)$ as follows
\[
\dist(L,\Sigma_m) := \min_{\tilde L\in \Sigma_m} \dist(L,\tilde L), \; \odist(\Sigma_m,L) := \min_{\tilde L\in \Sigma_m} \odist(\tilde L,L).
\]
Observe that if $L\cap \interior(K) \ne \emptyset$ then  
\[
\dist(L,\Sigma_m) = \min\{\dist(L,\tilde L): \tilde L^\perp \cap K^* \ne \{0\}\}.
\]
Similarly, if $L^\perp\cap \interior(K^*) \ne \emptyset$  then  
\[
\odist(\Sigma_m,L) = \min\{\odist(\tilde L,L): \tilde L \cap K \ne \{0\}\}.
\]
The above distances to ill-posedness can alternatively be seen as {\em radii of well-posedness:}  If $L \in \Gr(V,m)\setminus \Sigma_m$ then for all $\tilde L \in \Gr(V,m)$
\[
\dist(L,\tilde L)<\dist(L,\Sigma_m) \Rightarrow \tilde L \not \in \Sigma_m
\]
and 
\[
\odist(\tilde L,L)<\odist(\Sigma_m,L) \Rightarrow \tilde L \not \in \Sigma_m.
\]
\subsection{Least violated trial solutions}
Define $\nu,\bar \nu:\Gr(V,m)\rightarrow \R$ as follows
\[
\nu(L):=
\min_{u\in K^*, y\in L^\perp \atop \vertiii{u}^*=1} \|y-u\|^*, \; \; 
\bar \nu(L):=\min_{v\in K, x\in L \atop \|x\|=1} \vertiii{v-x}.
\]
Observe that $L\cap \interior(K) \ne \emptyset$ if and only if  $\nu(L) >0.$ 
In this case $\nu(L)$ can be seen as a measure of a {\em least violated} trial point for the alternative problem $u \in L^\perp \cap K^*\setminus\{0\}.$  Similarly, $L^\perp\cap \interior(K^*) \ne \emptyset$ if and only if  $\bar\nu(L) >0.$ 
In this case $\bar\nu(L)$ can be seen as a measure of a {\em least violated} trial point for the alternative problem $x \in L \cap K\setminus\{0\}.$ 
In Section~\ref{sec.induced.norm} below we show that for a suitable norm $\vertiii{\cdot}$ in $V$, $\nu(L)$ can also be seen as a measure of the {\em most interior} normalized solution to the feasibility problem $x \in L\cap \interior(K)$ when this problem is feasible.

The next result generalizes~\cite[Theorem 1]{CheuCP10} as well as \cite[Proposition 1.6]{AmelB12}:

\begin{theorem}\label{thm:Grassmann.nu} 
Assume $L\in \Gr(V,m)$.
\begin{description} 
\item[(a)]
If $L \cap \interior(K) \neq \emptyset$ then
$\dist(L,\Sigma_m) = \nu(L).$
\item[(b)] If $L^\perp \cap \interior(K^*) \neq \emptyset$ then $\odist(\Sigma_m,L) = \bar\nu(L).$
\end{description}
\end{theorem}

The proof of Theorem~\ref{thm:Grassmann.nu} relies on the properties of $\dist(\cdot,\cdot)$ and $\nu(\cdot)$ stated in Proposition~\ref{prop.dist} and Proposition~\ref{prop.nu} below.  These propositions in turn are consequences of the following technical lemma.

\begin{lemma}\label{lemma.min.norm} Assume $L\in \Gr(V,m)$ and $p\in V$. Then
\begin{equation}\label{eq.min.norm} 
\min_{z\in L} \|p-z\| = \max_{q\in L^\perp\atop \|q\|^* = 1}\ip{q}{p}.
\end{equation}
Furthermore, $\bar z \in L$ and $\bar q \in L^\perp$ with $\|\bar q\|^*=1$ attain the  optimal values in \eqref{eq.min.norm} if and only if $\|p - \bar z\| = \ip{\bar q}{ p - \bar z} =  \ip{\bar q}{p}.$
\end{lemma}
\begin{proof} By properties of norms and convex duality we have
\[
\min_{z\in L} \|p-z\| = \min_{z\in L}\max_{q\in V\atop \|q\|^* \le 1}\ip{q}{p-z}
=
\max_{q\in V\atop \|q\|^* \le 1} \min_{z\in L}\ip{q}{p-z} = 
\max_{q\in L^\perp\atop \|q\|^* = 1}\ip{q}{p}.
\]
This shows~\eqref{eq.min.norm}. 
The second statement follows from \eqref{eq.min.norm} and the the fact that  
\[
\|p - \bar z\| \le \ip{\bar q}{ p - \bar z} =  \ip{\bar q}{p}
\]
for $\bar z \in L$ and $\bar q \in L^\perp$ with $\|\bar q\|^*=1$.
\end{proof}

\begin{proposition}\label{prop.dist} For $L_1,L_2\in \Gr(V,m)$
\[
\dist(L_1,L_2) = \max_{x\in L_1\atop \|x\|=1} \min_{v\in L_2} \vertiii{x-v} = 
\max_{x\in L_1,u\in L_2^\perp \atop \|x\| = 1, \vertiii{u}^*=1} \ip{u}{x} 
= \max_{u\in L_2^\perp\atop \vertiii{u}^*=1} \min_{y\in L_1^\perp} \|u-y\|^*.
\]
Furthermore, $\bar x\in L_1, \bar v\in L_2, \bar u\in L_2^\perp, \bar y\in L_1^\perp$ attain the above optimal values if and only if $\|\bar x\|=1, \vertiii{\bar u}^* = 1, 
\ip{\bar u}{\bar x-\bar v} = \ip{\bar u}{\bar x} = \vertiii{\bar x-\bar v},$ and 
$\ip{\bar x}{\bar u-\bar y} = \ip{\bar u}{\bar x} = \|\bar u -\bar y\|^*.$
\end{proposition}
\begin{proof}
This follows by applying  Lemma~\ref{lemma.min.norm} to each of the norm minimization problems $\displaystyle \min_{v\in L_2} \vertiii{x-v}$ and $\displaystyle\min_{y\in L_1^\perp} \|u-y\|^*$.
\end{proof}

\begin{proposition}\label{prop.nu} 
If $L\in \Gr(V,m)$ and $L\cap \interior(K)\ne \emptyset$ then
\[
\nu(L) = 
\min_{u\in K^*, y\in L^\perp \atop \vertiii{u}^*=1} \|u-y\|^*=\min_{u\in K^* \atop \vertiii{u}^*=1} \max_{x\in L\atop \|x\|=1} \ip{u}{x}.
\]
Furthermore,  $\bar u\in K^*, \bar y\in L^\perp, \bar x\in L$ attain the above optimal values if and only if $ \vertiii{\bar u}^*=1, \|\bar x\| = 1, 
\ip{\bar u -\bar y}{\bar x} = \ip{\bar u}{\bar x} = \|\bar u-\bar y\|^*.$ 
\end{proposition}
\begin{proof}
This follows by applying Lemma~\ref{lemma.min.norm} to the norm minimization problem $\displaystyle\min_{y\in L^\perp} \|u-y\|^*$.
\end{proof}

\newpage
\noindent
{\em Proof of Theorem~\ref{thm:Grassmann.nu}.}
\begin{description}
\item[(a)]
First we show $\dist(L,\Sigma_m) \ge \nu(L).$ To do so, it suffices to show that  $\dist(L,\tilde L) \ge \nu(L)$ for all $\tilde L \in \Gr(V,m)$ with $\tilde L^\perp \cap K^* \ne \{0\}$.   Assume $\tilde L \in \Gr(V,m)$ is such that $\tilde L^\perp \cap K^* \ne \{0\}$.  Let $\bar u \in \tilde L^\perp \cap K^*$ with $\vertiii{\bar u}^* = 1.$  By Proposition~\ref{prop.dist}, it follows that
\begin{multline*}
\nu(L) = \min_{u\in K^*, y\in L^\perp \atop \vertiii{u}^*=1} \|u-y\|^* \le 
 \min_{y\in L^\perp} \|\bar u-y\|^* \le  
\max_{u\in \tilde L^\perp \atop \vertiii{u}^* =1 } \min_{y\in L^\perp} \|u-y \|^* \\ =  
 \dist(L,\tilde L).\end{multline*}

Next we show $\dist(L,\Sigma_m) \le \nu(L).$   To do so, it suffices to show that there exists $\tilde L \in \Gr(V,m)$ such that $\tilde L^\perp \cap K^* \ne \{0\}$ and $\dist(L,\tilde L) \le \nu(L).$   For simplicity, we write $\nu$ as shorthand for $\nu(L)$ is the rest of this proof.  
By Proposition~\ref{prop.nu} there exist $\bar x\in L, \bar u\in K^*,$ and $\bar y\in L^\perp$ such that $\vertiii{\bar u}^*=1, \|\bar x\|=1$ and 
$\nu= \ip{\bar u - \bar y}{\bar x}=\ip{\bar u }{\bar x}=\|\bar y-\bar u\|^*$.  Let $H:= \{x \in V: \ip{\bar u - \bar y}{x} = 0\}.$  Observe that $\bar x \not \in H$ because 
$\ip{\bar u - \bar y}{\bar x} = \|\bar y - \bar u\|^* > 0.$  Thus $L = \lspan(\{\bar x\} \cup (L \cap H))$.  Let $\bar v \in V$ be such that $\vertiii{\bar v} = 1$ and $\ip{\bar u}{\bar v}=\vertiii{\bar u}^*=1.$  Next, define
$$\tilde L:=\lspan(\{\bar x - \nu \bar v\}\cup(L \cap H)).$$
Since $\bar u - \bar y \in H^\perp \subseteq (H\cap L)^\perp$ and  $\bar y \in L^\perp \subseteq (H\cap L)^\perp$, we have $\bar u = \bar u - \bar y + \bar y \in (H\cap L)^\perp$; moreover, $\ip{\bar u}{\bar x - \nu \bar v} = \ip{\bar u}{\bar x} - \nu \ip{\bar u}{\bar v} = 0.$  Therefore $0 \ne \bar u \in \tilde L^\perp \cap K^*.$   To upper bound $\dist(L,\tilde L),$ assume $x\in L \setminus\{0\}.$  Then $x = \alpha \bar x + v$ for some $v\in H\cap L$.  We thus have
\[
\|x\| \ge \left|\ip{\frac{\bar y - \bar u}{\|\bar y - \bar u\|^*}}{x}\right| = |\alpha| \frac{\ip{\bar y - \bar u}{\bar x}}{\|\bar y - \bar u\|^*} = |\alpha|.
\] 
The second step follows because $\bar y - \bar u\in L^\perp.$  Put $\tilde x := \alpha(\bar x - \nu \bar v) + v$ and observe that $\vertiii{\tilde x - x} = 
|\alpha \nu| \vertiii{\bar v} = |\alpha| \nu.$ Hence
\[
\frac{\vertiii{\tilde x - x}}{\|x\|} = \frac{|\alpha| \nu}{\|x\|} \le \nu.
\]
Since this construction can be repeated for all $x\in L\setminus\{0\}$, it follows that
$$
\dist(L,\tilde L) = \max_{x\in L\atop x\ne 0} \min_{\tilde x \in \tilde L}  \frac{\vertiii{\tilde x - x}}{\|x\|}\le \nu.
$$

\item[(b)]
This follows via a similar, though simpler, reasoning to that used in part (a).  First, we show that $\odist(\Sigma_m,L) \ge \bar\nu(L).$ To do so, it suffices to show that  $\dist(\tilde L,L) \ge \bar\nu(L)$ for all $\tilde L \in \Gr(V,m)$ with $\tilde L \cap K \ne \{0\}$.   Assume $\tilde L \in \Gr(V,m)$ is such that $\tilde L \cap K \ne \{0\}$.  Let $\bar v \in \tilde L \cap K$ with $\bar v\ne 0.$ Thus
\begin{multline*}
\bar\nu(L) = \min_{v\in K, x\in L \atop \|x\|=1} \vertiii{x-v} \le 
\min_{x\in L \atop x\neq 0} \frac{\vertiii{x-\bar v}}{\|x\|} \le  
\max_{v\in \tilde L \atop v\ne 0 } \inf_{x\in L\atop x\ne 0} \frac{\vertiii{x-v}}{\|x\|} \\
= \odist(\tilde L,L).
 \end{multline*}
Next we show $\odist(\Sigma_m,L) \le \bar\nu(L).$   To do so, it suffices to show that there exists $\tilde L \in \Gr(V,m)$ such that $\tilde L \cap K \ne \{0\}$ and $\odist(\tilde L,L) \le \bar\nu(L).$   For simplicity, we write $\bar \nu$ as shorthand for $\bar \nu(L)$ is the rest of this proof.  
Let $\bar x\in L$ and $\bar v\in K$ be such that $\|\bar x\|=1$ and 
$\bar \nu= \vertiii{\bar x - \bar v}^*$. If $\bar v = 0$ then from \eqref{eq.upperbounds.dist} it readily follows that $\bar \nu = \min_{x\in L,\|x\|=1}\vertiii{x} \ge \odist(\tilde L,L)$ for any $\tilde L \in \Sigma_m$ and there is nothing to show.  Thus assume $\bar v \ne 0$.
Let $\bar y\in V$ be such that $\|\bar y\|^* = 1$ and $\ip{\bar y}{\bar x} = \|\bar x\|=1$. Let $H:= \{x \in V: \ip{\bar y}{x} = 0\}.$  Observe that $\bar x \not \in H$ because $\ip{\bar y}{\bar x} = 1.$  Thus $L = \lspan(\{\bar x\} \cup (L \cap H))$.   Next, define
$$\tilde L:=\lspan(\{\bar v\}\cup(L \cap H)).$$
By construction, we have $0\ne \bar v \in \tilde L \cap K.$ To upper bound $\odist(\tilde L,L),$ assume $\tilde x\in \tilde L \setminus\{0\}.$  Then $\tilde x = \alpha \bar v + v$ for some $v\in H\cap L$.  Put $x:=\alpha \bar x + v$ and observe that $
\|x\| \ge |\ip{\bar y}{x}| = |\alpha|
$ since $\|\bar y\|^*=1$ and $\ip{\bar y}{v} = 0$. Therefore when $\alpha \neq 0$, 
\[
\frac{\vertiii{\tilde x - x}}{\|x\|} = \frac{|\alpha| \vertiii{\bar x - \bar v}}{\|x\|} \le \frac{|\alpha|\bar \nu}{|\alpha|} \le \bar \nu. 
\]
On the other hand, when $\alpha = 0$ 
\[
\frac{\vertiii{\tilde x - x}}{\|x\|} =\frac{0}{\|v\|} = 0 \le \bar \nu.
\]
Since this construction can be repeated for all $\tilde x\in \tilde L\setminus\{0\}$, it follows that $$\odist(\tilde L,L)  = \max_{\tilde x \in \tilde L \atop \tilde x \ne 0} \inf_{x\in L \atop x\ne 0} \frac{\vertiii{\tilde x - x}}{\|x\|} \le \bar \nu.$$

\end{description}
\qed

\subsection{Renegar's distance to ill-posedness}
\label{sec.renegar}
We now relate the condition measures $\dist(\cdot), \odist(\cdot)$ with the now classical Renegar's distance to ill-posedness.  The key conceptual difference between Renegar's approach and the Grassmann approach is that Renegar~\cite{Rene95a,Rene95}  considers conic feasibility problems where the linear spaces $L$ and $L^\perp$ are defined via a linear mapping  $A:W\rightarrow V.$ We next present a variation (more precisely a slight extension) of Renegar's distance to ill-posedness.  

Assume $W$ is a finite dimensional Euclidean vector space with norm $| \cdot |$ and $\dim(W) < \dim(V)$.  For  a linear mapping $A:W\rightarrow V$ 
consider the conic systems \eqref{primal} and \eqref{dual} defined by taking $L=\Image(A)$. 
These two conic systems can respectively be written as 
\begin{equation}\label{renegar.primal}
Ax \in K \setminus\{0\}
\end{equation}
and \begin{equation}\label{renegar.dual} 
 A^*u = 0, \; u\in K^*  \setminus\{0\}.
\end{equation}
Here $A^*:V\rightarrow W$ denotes the {\em adjoint} operator of $A$, that is, the linear mapping satisfying $\ip{y}{Aw} = \ip{A^*y}{w}$ for all $y\in V, w\in W.$

Let $\mathcal L(W,V)$ denote the set of linear mappings from $W$ to $V$.  Define the operator norms $ \|\cdot\|, \vertiii{\cdot}: \mathcal L(W,V) \rightarrow \R $ as follows
\[
\|A\|:=\max_{w\in W\atop |w| = 1} \|Aw\| \; \text{ 
and } \; 
\vertiii{A} := \max_{w\in W\atop |w| = 1} \vertiii{Aw}.
\]

We say that $A \in \mathcal L(W,V)$ is {\em well-posed} if either \eqref{renegar.primal} or \eqref{renegar.dual} is feasible and remains so for sufficiently small  perturbations of $A$.  Let $\Sigma \subseteq \mathcal L(W,V)$ denote the set of {\em ill-posed mappings,} that is, mappings that are not well-posed.
It is easy to see that if $A\in \mathcal L(W,V) \setminus \Sigma$ then 
either $\Image(A)\cap \interior(K) \ne \emptyset$ or $\ker(A^*)\cap \interior(K^*) \ne \emptyset$.

Renegar's distance to ill-posedness $\Rdist:\mathcal L(W,V) \rightarrow \R$ is defined as:
\[
\Rdist(A,\Sigma) := \min\left\{\vertiii{A - \tilde A}: \tilde A \in \Sigma\right\}.
\]
This definition is a slight extension of Renegar's original definition in~\cite{Rene95a,Rene95} due to our use of the two norms $\|\cdot\|, \vertiii{\cdot}$ in $V$.  Renegar's original distance to ill-posedness corresponds to the case $\|\cdot\| = \vertiii{\cdot}$.

Renegar's distance to ill-posedness $\Rdist(A,\Sigma)$ can alternatively be defined as the {\em distance to infeasibility} of \eqref{renegar.primal} or \eqref{renegar.dual}.  More precisely, if $A\in \mathcal L(W,V)$ is such that \eqref{renegar.primal} is feasible, then
\[
\Rdist(A,\Sigma) := \min\left\{\vertiii{A - \tilde A}: \tilde A x \in K\setminus\{0\} \; \text{ is infeasible}\right\}.
\]
Likewise if  $A\in \mathcal L(W,V)$ is such that \eqref{renegar.dual} is feasible, then
\[
\Rdist(A,\Sigma) := \min\left\{\vertiii{A - \tilde A}: \tilde A^* u=0, \; u \in K^*\setminus\{0\} \; \text{ is infeasible}\right\}.
\]
Observe that if $A\in \mathcal L(W,V) \setminus \Sigma$ and $\Image(A)\cap \interior(K) \ne \emptyset$ then 
\[
\Rdist(A,\Sigma) = \min\left\{\vertiii{A-\tilde A}: \tilde A^* u = 0 \; \text{ for some } u \in K^*\setminus \{0\}\right\}.
\]
On the other hand, if  $A\in \mathcal L(W,V) \setminus \Sigma$ and $\Image(A)^\perp \cap \interior(K^*) \ne \emptyset$ then  
\[
\Rdist(A,\Sigma) = \min\left\{\vertiii{A-\tilde A}: \tilde A x \in  K \; \text{ for some } x \in W\setminus \{0\}\right\}.
\]
We shall rely on one more piece of notation.  Given $A  \in \mathcal L(W,V)$  the mapping $A^{-1}: \Image(A)\rightarrow W$ is well defined provided $A$ is injective.  In this case let $\|A^{-1}\|$ denote the following operator norm
\[
\|A^{-1}\|:=\max_{x\in \Image(A)\atop \|x\| = 1} |A^{-1}x|.
\]

The following result gives a general version of \cite[Theorem 1.4]{AmelB12}.  We note that the  statements of these theorems are nearly identical. 
\begin{theorem} 
\label{thm.renegar}
Assume $A \in \mathcal L(W,V)\setminus \Sigma$ is injective and $L := \Image(A)\in \Gr(V,m).$ 

\begin{description} 
\item[(a)] 
If $L \cap \interior(K) \ne \emptyset$ then
\[
\frac{1}{\|A\|} \le \frac{\dist(L,\Sigma_m)}{\Rdist(A,\Sigma)} \le \|A^{-1}\|,
\]
or equivalently
\[
 \frac{1}{\dist(L,\Sigma_m)} \le \frac{\|A\|}{\Rdist(A,\Sigma)}\le  \frac{\kappa(A)}{\dist(L,\Sigma_m)}
\]
where $\kappa(A) = \|A\| \cdot \|A^{-1}\|.$ 
\item[(b)] 
If $L^\perp \cap \interior(K^*) \ne \emptyset$ then
\[
\frac{1}{\|A\|} \le \frac{\odist(\Sigma_m,L)}{\Rdist(A,\Sigma)} \le \|A^{-1}\|,
\]
or equivalently
\[
 \frac{1}{\odist(\Sigma_m,L)} \le \frac{\|A\|}{\Rdist(A,\Sigma)}\le  \frac{\kappa(A)}{\odist(\Sigma_m,L)}.
\]
\end{description}
\end{theorem}
\begin{proof}
\begin{description}
\item[(a)]
First, we prove $\Rdist(A,\Sigma) \le \nu(L) \|A\| =  \dist(L,\Sigma_m) \|A\|.$  To that end, assume $\bar u \in K^*, \vertiii{\bar u}^* =1$ is such that
$
\nu(L) = \displaystyle\max_{x\in L \atop \|x\|=1} \ip{\bar u}{x}
$ as in Proposition~\ref{prop.nu}.
Then
\[
\nu(L) \|A\| \ge \max_{w\in W \atop |w|=1} \ip{\bar u}{Aw} = |A^*\bar u|^*.
\]
Let $\bar v\in V, \vertiii{\bar v} = 1$ be such that $\ip{\bar u}{\bar v} = \vertiii{\bar u}^* =1$.  
Now construct $\Delta A: W \rightarrow V$ as follows
\[
\Delta A(w):= - \ip{A^* \bar u}{w} \bar v.
\]
Observe that $\vertiii{\Delta A} = |A^*u|^* \cdot \vertiii{\bar v} \le \nu(L)\|A\|,$ and $\Delta A^*: V \rightarrow W$ is defined by
\[
\Delta A^*(y) = - \ip{y}{\bar v} A^*\bar u.
\]
In particular $(A+ \Delta A)^*\bar u = A^* \bar u - \ip{\bar u}{\bar v} A^*\bar u = 0$ and $\bar u \in K^*\setminus \{0\}.$  Therefore
\[
\Rdist(A,\Sigma_m) \le \vertiii{\Delta A} \le \nu(L)\|A\| = \dist(L,\Sigma_m) \|A\|.
\]
For the second inequality, assume  
$\tilde A\in \mathcal L(W,V)$ is such that $\ker(\tilde A^*) \cap K^*\setminus\{0\} \ne \emptyset.$   By performing a slight perturbation on $\tilde A$ if needed, we may assume that $\tilde L \in \Gr(V,m)$.  Thus
\[
\dist(L,\Sigma_m) \le \dist(L,\tilde L).
\]
 Assume $\bar x \in L$ with $\|\bar x\|=1$ is such that 
\[
\dist(L,\tilde L) = \min_{\tilde x\in \tilde L} \vertiii{\bar x-\tilde x}. 
\]
Let $\bar w \in W$ be such that $A \bar w = \bar x$. Since $\tilde A \bar w \in \tilde L$, it follows that
\[
\dist(L,\tilde L) \le  \vertiii{(A - \tilde A)\bar w} \le  \vertiii{(A - \tilde A)}\cdot |\bar w|.
\]
Thus 
\[
\vertiii{A - \tilde A} \ge \frac{\dist(L,\tilde L)}{|w|} \ge \frac{\dist(L,\Sigma_m)}{\|A^{-1}\|}.
\]
Since this holds for all $\tilde A$ with $\ker(\tilde A) \cap K^*\setminus\{0\}$, we conclude that
\[
\Rdist(A,\Sigma_m) \ge \frac{\dist(L,\Sigma_m)}{\|A^{-1}\|}.
\]
\item[(b)]
First, we prove $\Rdist(A,\Sigma_m) \le \bar\nu(L) \|A\| =  \odist(\Sigma_m,L) \|A\|.$  To that end, assume $\bar v \in K, \bar x \in L$ are such that $\|\bar x\|=1$ and $\bar \nu = \vertiii{\bar x - \bar v}.$  

Let $\bar w := A^{-1}\bar x$ and $\bar z \in W$ be such that $|\bar z|^* = 1, \; \ip{\bar z}{\bar w} = |\bar w|$.  Observe that $|\bar w| \ge \frac{1}{\|A\|}$ since $1=\|x\| = \|A\bar w\| \le \|A\| |\bar w|.$  Define $\Delta A:W\rightarrow V$ as follows
\[
\Delta A(w) := (\bar v - \bar x) \frac{\ip{\bar z}{w}}{|\bar w|}.
\]
Observe that $(A+\Delta A)\bar w = \bar v \in K$ and 
$\vertiii{\Delta A} = \frac{\vertiii{\bar v - \bar x}}{|\bar w|} \le \bar \nu \|A\|$.  Thus $\Rdist(A,\Sigma_m) \le \vertiii{\Delta A}\le \bar\nu(L) \|A\|=  \odist(\Sigma_m,L) \|A\|.$
For the second inequality assume  assume  
$\tilde A\in \mathcal L(W,V)$ is such that  $\Image(\tilde A) \cap K \setminus\{0\}\ne \emptyset. $  By performing a slight perturbation of $\tilde A$ if needed, we may assume that $\tilde L:=\Image(\tilde A) \in \Gr(V,m)$.  Thus
\[
\odist(\Sigma_m,L) \le \odist(\tilde L,L).
\]
Assume $0\ne \tilde x \in \tilde L$ is such that 
\[
\odist(\tilde L,L) = \inf_{x\in L}\frac{\vertiii{\tilde x-x}}{\| x\|}. 
\]
Let $ w \in W$ be such that $\tilde A  w = \tilde x$. Since $Aw \in L$ we have
\[
\odist(\tilde L,L) \le \frac{\vertiii{\tilde A w-Aw}}{\|Aw\|} \le 
\frac{\vertiii{\tilde A - A} |w|}{\|Aw\|}. 
\]
Hence
\[
\vertiii{\tilde A - A} \ge  \frac{\odist(\tilde L,L) \|Aw\|}{|w|} 
= \frac{\odist(\tilde L,L) \|Aw\|}{|A^{-1}(Aw)|} \ge \frac{\odist(\Sigma_m,L)}{\|A^{-1}\|}.\]
Since this holds for all $\tilde A$ with $\Image(\tilde A) \cap K\setminus\{0\}\ne \emptyset$,   we conclude that
\[
\Rdist(A,\Sigma_m) \ge \frac{\odist(\Sigma_m,L)}{\|A^{-1}\|}.
\]

\end{description}
\end{proof}



\section{Euclidean and induced norms}
\label{sec.norms}

We next discuss the special but important choices of norms, namely the Euclidean norm defined by the inner product in $V$ and a certain {\em induced norm} defined by the cone $K$ and a point $e\in \interior(K)$. 
\subsection{Euclidean norm}
\label{sec.euclidean.norm}

The Euclidean norm $\|\cdot\|_2$ is the norm defined by the inner product in $V$.  That is, for $v\in V$
\[
\|v\|_2 := \sqrt{\ip{v}{v}}.
\]
The Euclidean norm is naturally related to angles.  Given $x,y\in V\setminus\{0\}$ define $\angle(x,y) := \arccos \frac{\ip{x}{y}}{\|x\|_2\|y\|_2} \in [0,\pi]$.  Given a linear subspace $L\subseteq V$ and a closed convex cone $C\subseteq V$ define
\[
\angle(L,C) := \min\{\angle(x,v): x \in L\setminus\{0\} \; v \in C\setminus\{0\}\} \in [0,\pi/2].
\]

\begin{proposition}  \label{prop.euclidean} Assume $\|\cdot\| = \vertiii{\cdot} = \|\cdot\|_2.$ Then 
\begin{description}
\item[(a)]
For $L_1,L_2\in \Gr(V,m)$
\[
 \|\Pi_{L_1} - \Pi_{L_2}\|_2 = \dist(L_1,L_2) = \odist(L_1,L_2) 
\]
where $\Pi_{L_i}:V \rightarrow L_i$ is the  orthogonal projection onto $L_i$ for $i=1,2$ and $\|\Pi_{L_1} - \Pi_{L_2}\|_2 = \max_{\|x\|_2=1} \|(\Pi_{L_1} - \Pi_{L_2})x\|_2.$ 
\item[(b)] If $L\in \Gr(V,m)$ is such that $L\cap \interior(K) \ne \emptyset$ then 
\[
\nu(L) = \sin\angle(L^\perp,K^*).
\]
\item[(c)] If $L\in \Gr(V,m)$ is such that $L^\perp\cap \interior(K^*) \ne \emptyset$ then 
\[
\bar \nu(L) = \sin\angle(L,K).
\]
\end{description}
\end{proposition}
\begin{proof}
\begin{description}
\item[(a)] 
A classical result from geometry (see, e.g., \cite[Section 12.4]{GoluV96}) states that $\|\Pi_{L_1} - \Pi_{L_2}\|_2 = \sin(\theta)$ where $\theta$ is the {\em largest principal angle} between $L_1$ and $L_2$, namely
\[
\theta = \max_{x\in L_1\atop x\ne 0} \min_{v \in L_2\atop v \ne 0} \angle(x,v) 
\in [0,\pi/2].
\]
Since $\theta \in [0,\pi/2]$ we have
\begin{align*}
\|\Pi_{L_1} - \Pi_{L_2}\|_2 
&= \max_{x\in L_1\atop x\ne 0} \min_{v \in L_2\atop v \ne 0} \sin(\angle(x,v)) \\
&= \max_{x\in L_1\atop x\ne 0} \min_{v \in L_2} \frac{\|x-v\|_2}{\|x\|_2}\\
& = \dist(L_1,L_2),
\end{align*}
and 
\begin{align*}
\|\Pi_{L_1} - \Pi_{L_2}\|_2 
&= \max_{x\in L_1\atop x\ne 0} \min_{v \in L_2\atop v \ne 0} \sin(\angle(x,v)) 
\\
&= \max_{x\in L_1\atop x\ne 0} \inf_{v \in L_2\atop v\ne 0} \frac{\|x-v\|_2}{\|v\|_2} \\
&= \odist(L_1,L_2).
\end{align*}
\item[(b)] 
Since $\angle(L^\perp,K^*) \in [0,\pi/2]$ we have
\[
\sin\angle(L^\perp,K^*) = \min_{u\in K^*,y\in L^\perp \atop u,y\ne 0}  \sin \angle(y,u) 
= \min_{u\in K^*,y\in L^\perp\atop \|u\|_2=1}  \|u-y\|_2 
= \nu(L).
\]
\item[(c)] 
This is nearly identical to part (b):  Since $\angle(L,K) \in [0,\pi/2]$ we have
\[
\sin\angle(L,K) = \min_{v\in K,x\in L\atop v,x\ne 0}  \sin \angle(x,v) 
= \min_{v\in K,x\in L\atop \|x\|_2=1}  \|v-x\|_2 
= \bar\nu(L).
\]
\end{description}
\end{proof}

Proposition~\ref{prop.euclidean} readily implies the symmetry of $\dist$ and $\odist$.  That is,  $\dist(L_1,L_2) = \dist(L_2,L_1)$ for all $L_1,L_2 \in \Gr(V,m)$, and likewise for $\odist(\cdot)$ if $\|\cdot\| = \vertiii{\cdot} = \|\cdot\|_2$.  Furthermore, Proposition~\ref{prop.euclidean} also implies 
that in this case $\dist(\cdot)$ coincides with the Grassmann distance 
used in \cite{AmelB12}.  Hence for this special choice of norms, Theorem~\ref{thm.renegar} subsumes \cite[Theorem 1.4]{AmelB12}.  Similarly, for this special choice of norms, Theorem~\ref{thm:Grassmann.nu} and Proposition~\ref{prop.euclidean} subsume \cite[Proposition 1.6]{AmelB12}.

\subsection{Induced norm and induced eigenvalue mappings}
\label{sec.induced.norm}
Assume $K \subseteq V$ is a regular closed convex cone and $e\in \interior(K)$.  We next describe a norm $\|\cdot\|_e$ in $V$ and a mapping $\lambda_{e}:V\rightarrow \R$ induced by the pair $(K,e)$.  These norm and mapping yield a natural interpretation of $\nu(L)$ as a measure of the {\em most interior} normalized solution to the feasibility problem $x \in L\cap \interior(K)$ when this problem is feasible.  

Define the norm $\|\cdot\|_e$ in $V$ induced by $(K,e)$ as follows (see \cite{CCPMultifold})
\[
\|x\|_e:=\min\{\alpha \ge 0: x+\alpha e \in K, \; -x+\alpha e \in K\}.
\]
Define the {\em eigenvalue} mapping $\lambda_{e}:V \rightarrow \R$ induced by $(K,e)$ as follows
\[
\lambda_{e}(x):=\max\{t\in \R: x-te \in K\}.
\]
Observe that $x \in K \Leftrightarrow \lambda_{e}(x) \ge 0$ and $x \in \interior(K) \Leftrightarrow \lambda_{e}(x) > 0.$  Furthermore, observe that when $x \in K$
\[
\lambda_e(x) = \max\{r\ge 0: \|v\|_e \le r \Rightarrow x+v\in K\}.
\]
Thus for $x\in K$, $\lambda_e(x)$ is a measure of how interior $x$ is in the cone $K$.  

It is easy to see that  $\|u\|_e^* = \ip{u}{e}$ for $u\in K^*$.  In analogy to the standard simplex, let
\[
\Delta(K^*):=\{u\in K^*: \|u\|_e^*=1\} = \{u\in K^*:\ip{u}{e}=1\}.
\]
It is also easy to see that the  eigenvalue mapping $\lambda_{e}$ has the following alternative expression
\[
\lambda_{e}(x) = \min_{u\in \Delta(K^*)} \ip{u}{x}.
\]
The following result readily follows from Proposition~\ref{prop.nu} and convex duality.

\begin{proposition}\label{prop.eigenvalue}
Let $K \subseteq V$ be a regular closed convex cone and $e\in \interior(K)$.  If $\vertiii{\cdot} = \|\cdot\|_e$, then for all $L\in \Gr(V,m)$
\[
\nu(L)= \min_{u\in \Delta(K^*)} \max_{x\in L\atop \|x\| \le 1} \ip{x}{u} = 
\max_{x\in L\atop \|x\| \le 1} \min_{u\in \Delta(K^*)} \ip{x}{u} = 
\max_{x\in L\atop \|x\| \le 1}  \lambda_{e}(x). 
\]
In particular, when $L \cap \interior(K) \ne \emptyset$ the quantity $\nu(L)$ can be seen as a measure of the {\em most} interior normalized point in $L\cap  \interior(K)$.
\end{proposition}

It is illustrative to further specialize Proposition~\ref{prop.eigenvalue} to two important cases.  The first case is $V=\R^n$ with the usual dot inner product, $K = \R^n_+$ and $e = \matr{1 & \cdots & 1}\transp \in \R^n_+$.  In this case $ \|\cdot\|_e= \|\cdot\|_\infty, \, \|\cdot\|_e^* = \|\cdot\|_1$, $(\R^n_+)^*=\R^n_+$ and $\Delta(\R^n_+)$ is the standard simplex 
$\Delta_{n-1}:=\{x\in \R^n_+:\sum_{i=1}^n x_i = 1\}$.   In this case $\lambda_{e}(x) = \displaystyle\min_{i=1,\dots,n} x_i$. Hence for $\vertiii{\cdot} = \|\cdot\|_e$ we have
\begin{equation}\label{eqn.nu.lp}
\nu(L)= \max_{x\in L\atop \|x\|\le1} \min_{j=1,\dots,n} x_j.
\end{equation}

The second special case is $V=\S^n$ with the trace inner product, $K = \S^n_+$ and $e = I \in \S^n_+$.  In this case $\|\cdot\|_e $ and $\|\cdot\|_e^*$ are respectively the operator norm and the nuclear norm in $\S^n$.  More precisely
\[
\|X\|_e = \max_{i=1,\dots,n} |\lambda_i(X)|, \; \|X\|_e^* = \sum_{i=1}^n |\lambda_i(X)|,
\]
where  $\lambda_i(X),\; i=1,\dots,n$ are the usual eigenvalues of $X$.  Furthermore, $(\S^n_+)^*=\S^n_+$ and $\Delta(\S^n_+)$ is the {\em spectraplex} $\{X\in \S^n_+: \sum_{i=1}^n \lambda_i(X) = 1\}$.   In this case $\lambda_{e}(x) = \min_{j=1,\dots,n} \lambda_j(X).$ Thus, in nice analogy to \eqref{eqn.nu.lp}, for $\vertiii{\cdot} = \|\cdot\|_e$ we have
\begin{equation}\label{eqn.nu.sdp}
\nu(L)=\max_{X\in L\atop \|X\| \le 1}  \min_{j=1,\dots,n}\lambda_{j}(X).
\end{equation}
Section~\ref{sec.symm.cones} below details how \eqref{eqn.nu.lp} and \eqref{eqn.nu.sdp} extend to the more general case when $K$ is a symmetric cone.
\subsection{Sigma measure} 

The induced eigenvalue function discussed in Section~\ref{sec.induced.norm} above can be defined more broadly.   Given $v\in K\setminus\{0\}$ define $\lambda_v:V\rightarrow [-\infty,\infty)$ as follows 
\[
\lambda_v(x):=\max\{t: x- tv \in K\}.
\]
Define $\sigma: \Gr(V,m)\rightarrow [0,\infty)$ as follows
\begin{equation}\label{def.sigma.feas}
\sigma(L):=\min_{v\in K\atop \vertiii{v}=1}  \max_{x\in L \atop \|x\| \le 1}\lambda_v(x).
\end{equation}
As we show in Proposition~\ref{prop.ye.cond} below, $\sigma(L)$  can be seen as a generalization of the sigma measure introduced by Ye~\cite{Ye94}.  Observe that $L \cap \interior(K) \ne \emptyset$ if and only if $\sigma(L) > 0$ and in this case $$\sigma(L)=\displaystyle\min_{v\in K\atop \vertiii{v}=1}  \max_{x\in L\cap K \atop \|x\| = 1}\lambda_v(x).$$

As Corollary~\ref{corol.nu.sigma} below shows, when $L\cap \interior(K) \ne \emptyset$, the quantities $\sigma(L)$ and $\nu(L)$ are  closely related. 
To that end, we rely on the following analogous result to Proposition~\ref{prop.nu}. 

\begin{proposition}\label{prop.dual.sym}  
If $L\in \Gr(V,m)$ and $L\cap \interior(K)\ne \emptyset$ then
\begin{equation}\label{eq.sigma.norm}
\sigma(L) = 
\min_{v\in K \atop \vertiii{v}=1} \max_{x \in L \atop \|x\|=1} \lambda_v(x) = 
\min_{v\in K, y \in L^\perp, u\in K^* \atop \vertiii{v}=1,\ip{u}{v}=1} \|u-y\|^*.
\end{equation}
\end{proposition}
\begin{proof} Assume $v\in K$ is fixed. By the construction of $\lambda_v$, convex duality, and Lemma~\ref{lemma.min.norm} we have
\begin{align*}
\max_{x \in L \atop \|x\|= 1} \lambda_v(x)
&= \max_{x \in L, t\in \R \atop \|x\|\le 1, x-tv\in K} t \\&
= \max_{x \in L, t\in \R \atop \|x\|\le 1} \min_{u\in K^*} (t + \ip{u}{x-tv}) \\
&=  \min_{u\in K^*}\max_{x \in L, t\in \R \atop \|x\|\le 1} (t + \ip{u}{x-tv}) \\
& = \min_{u\in K^* \atop \ip{u}{v}=1} \max_{x\in L\atop \|x\|= 1} \ip{u}{x} \\
& = \min_{u\in K^*,y\in L^\perp \atop \ip{u}{v}=1} \|u-y\|^*.
\end{align*}
 We thus get \eqref{eq.sigma.norm} by taking minimum over the set $\{v\in K: \vertiii{v}=1\}$.
\end{proof}
\begin{corollary} 
\label{corol.nu.sigma} Assume $L\in \Gr(V,m)$ is such that $L\cap \interior(K)\ne \emptyset$.
\begin{description}
\item[(a)] For any two norms $\|\cdot\|, \vertiii{\cdot}$ in $V$ the following holds
\[
1\le \min_{v\in K, u\in K^* \atop \vertiii{v}=1,\ip{u}{v}=1} \vertiii{u}^* \le \frac{\sigma(L)}{\nu(L)} \le 
 \frac{1}{
 \displaystyle \min_{u\in K^*\atop \vertiii{u}^*=1} \max_{v\in K \atop \vertiii{v}=1} \ip{u}{v}}. 
\]
\item[(b)] If $\|\cdot\| = \vertiii{\cdot} = \|\cdot\|_2$ then
\[
1 \le \frac{\sigma(L)}{\nu(L)} \le \frac{1}{ \cos(\Theta(K^*,K))}.
\]
where 
\[
\Theta(K^*,K):=\max_{u\in K^*\setminus\{0\}}\min_{v\in K\setminus\{0\}} \angle(u,v).
\]
In particular, if in addition $K=K^*$  then
$\nu(L) = \sigma(L).$
\item[(c)] If $\vertiii{\cdot} = \|\cdot\|_e$ then
\[
\sigma(L) = \nu(L).
\] 

\end{description}
\end{corollary}
\begin{proof}
\begin{description}
\item[(a)] The first inequality is an immediate consequence of H\"older's inequality~\eqref{eq.holder}.  Next, from Proposition~\ref{prop.dual.sym} it follows that $\sigma(L) = \|\bar u - \bar y\|^*$ for some $\bar v \in K, \bar y\in L^\perp, \bar u\in K^*$ with $\vertiii{\bar v} = 1, \ip{\bar u}{\bar v} = 1$.  Thus from the construction  of $\nu(L)$ we get
$$ 
\nu(L) \le \frac{\|\bar u - \bar y\|^*}{\vertiii{\bar u}^*} \le \frac{\sigma(L)}{\displaystyle\min_{v\in K, u\in K^*\atop \vertiii{v}=1,\ip{u}{v}=1}\vertiii{u}^*}
$$ 
and hence the second inequality follows.

For the third inequality assume $\nu(L) = \|\hat u - \hat y\|^*$ for some $\hat u \in K^*, \hat y \in L^\perp$ with $\vertiii{\hat u}^* = 1.$  Then by Proposition~\ref{prop.dual.sym} we get
\begin{align*}
\sigma(L)& = \min_{v\in K, y \in L^\perp, u\in K^* \atop \vertiii{v}=1,\ip{u}{v}=1} \|u-y\|^* \leq \inf_{v\in K, y \in L^\perp, \atop \vertiii{v}=1, \ip{\hat u}{v}\neq 0 } \|\frac{\hat u}{\ip{\hat u}{v}}-y\|^*\\
 & = \inf_{v\in K, y \in L^\perp, \atop \vertiii{v}=1, \ip{u}{v}=1} \frac{\|\hat u-y\|^*}{\ip{\hat u}{v}}
 = \frac{\displaystyle\min_{y \in L^\perp} \|\hat u-y\|^*}{ \displaystyle\max_{v\in K\atop  \vertiii{v}=1}  \ip{\hat u}{v}}\\
 & \leq  \frac{\|\hat u-\hat y\|^*}{\displaystyle \max_{v\in K\atop  \vertiii{v}=1}  \ip{\hat u}{v}}, 
\end{align*}
hence
\[
\sigma(L) \le \frac{\|\hat u - \hat y\|^*}{\displaystyle\max_{v\in K \atop \vertiii{v}=1} \ip{\hat u}{v}} \le \frac{\nu(L)}{\displaystyle\min_{u\in K^*\atop \vertiii{u}^*=1}\max_{v\in K \atop \vertiii{v}=1} \ip{u}{v}}
\]
and the third inequality follows.
\item[(b)] The first inequality follows from part (a).  For the second inequality observe that since $\cos(\cdot)$ is decreasing in $[0,\pi]$
\begin{align*}
\cos(\Theta(K^*,K)) &= \min_{u\in K^*\setminus\{0\}}\max_{v\in K\setminus\{0\}} \cos(\angle(u,v)) \\
&= \min_{u\in K^*\setminus\{0\}}\max_{v\in K\setminus\{0\}} \frac{\ip{u}{v}}{\|u\|_2\cdot \|v\|_2} \\
&= \min_{u\in K^*\atop \|u\|_2 = 1}\max_{v\in K\atop \|v\|_2=1} \ip{u}{v}.
\end{align*}
The second inequality then follows from part (a) as well.

If in addition $K=K^*$ then $\Theta(K^*,K) = 0$ and consequently $\frac{\sigma(L)}{\nu(L)}=1$.

\item[(c)]  Since $\vertiii{\cdot} = \|\cdot\|_e$, we have $\vertiii{e} =1$ and $\vertiii{u}^* = \ip{u}{e}$  for all $u\in K^*$.  Thus
$\displaystyle\min_{u\in K^*\atop \vertiii{u}^*=1}\max_{v\in K \atop \vertiii{v}=1} \ip{u}{v} \ge \displaystyle\min_{u\in K^*\atop \vertiii{u}^*=1} \ip{u}{e} = 1.$
Therefore from part (a) it follows that $\frac{\sigma(L)}{\nu(L)} = 1.$ 
\end{description}

\end{proof}

The following example shows that the upper bound in Corollary~\ref{corol.nu.sigma}(b) is tight.  

\begin{example} Assume $\phi \in (0,\pi/4).$  Let $V = \R^2$ be endowed with the dot inner product and $K:=\{(x_1,x_2)\in V: \cos(\phi)x_2 \ge |x_1|\}.$ Then $K^*=\{(x_1,x_2)\in V: x_2 \ge \cos(\phi)|x_1|\}.$ In this case $\Theta(K,K^*) = \pi/2-2\phi.$  For $L = \{x\in V: x_1 = 0\}$ and $\|\cdot\| = \vertiii{\cdot} = \|\cdot\|_2$ it is easy to see that $ \nu(L) = \sin(\phi)$ and  $\sigma(L) = 1/(2\cos(\phi))$.  Hence
\[
\frac{\sigma(L)}{\nu(L)} = \frac{1}{2\sin(\phi)\cos(\phi)} = \frac{1}{\sin(2\phi)} = \frac{1}{\cos(\pi/2-2\phi)} = \frac{1}{\cos(\Theta(K,K^*))}.
\]  
\end{example}

\section{Symmetry measure}
\label{sec.symmetry.meas}
Next, we will consider the symmetry measure of $L$, which has been used as a measure of conditioning~\cite{BellF08,BellF09b}.  This measure is defined as follows.   Given a set $S$ in a vector space such that $0\in S$, define
\begin{equation}\label{eq.sym.point}
\sym(0,S):= \max\{t\ge 0: w\in S \Rightarrow -tw \in S\}.
\end{equation}
Observe that $\sym(0,S) \in [0,1]$ with $\sym(0,S)=1$ precisely when $S$ is perfectly symmetric around $0$.

Assume $A:V\rightarrow W$ is a linear mapping such that $L=\ker(A)$.  Define the {\em symmetry} measure of $L$ relative to $K$ as follows.   
\begin{equation}\label{eq.sym.set}
\sym(L):= \sym(0,A(\{ x\in K: \|x\|\le 1\})).
\end{equation}
It is easy to see that $\sym(L)$ depends only on $L,K$ and not on the choice of $A$. More precisely, $\sym(0,A(\{ x\in K: \|x\|\le 1\})) = 
\sym(0,A'(\{ x\in K: \|x\|\le 1\}))$ if $\ker(A) = \ker(A') = L$.
Indeed, the quantity $\sym(L)$ can be alternatively defined  in terms of $L$ and $K$ alone with no reference to any linear mapping $A$ as the next proposition states.
\begin{proposition}\label{prop.symmetry} Assume $L\in \Gr(V,m)$.  Then
\[
\sym(L):=\min_{v\in K\atop \|v\|=1} \max_{x\in K, t \in \R} 
\{t: x+tv\in L, \; \|x\|\le 1\}.
\]
\end{proposition}
\begin{proof}
Assume $A:V\rightarrow W$ is such that $L=\ker(A)$.  From \eqref{eq.sym.point} and~\eqref{eq.sym.set} it follows that for $S:= \{Ax: x\in K, \|x\|\le 1\}$
\begin{align*}
\sym(L) &= \min_{v\in K\atop \|v\|\le 1} \max_{t\in \R}\{t: -tAv \in S\} \\
&= \min_{v\in K\atop \|v\|\le 1} \max_{x\in K,\,t\in \R}\{t: -tAv = Ax, \; \|x\| \le 1\} \\
&= 
\min_{v\in K\atop \|v\|\le 1} \max_{x\in K,\,t\in \R}\{t: x+tv \in L, \; \|x\| \le 1\}.
\end{align*}
\end{proof}
Observe that $L \cap \interior(K) \ne \emptyset$ if and only if $\sym(L) > 0$. It is also easy to see that $\sym(L) \in [0,1]$ for all $L\in \Gr(V,m)$.  The following result is a general version of \cite[Proposition 22]{EpelF02}.

\begin{theorem} 
\label{thm.symmetry}
Assume $L\in \Gr(V,m)$ is such that $L \cap \interior K \neq \emptyset$, $\|\cdot \| = \vertiii{\cdot}$ and $\sym(L)<1$.  Then
\[
\frac{\sym(L)}{1+\sym(L)}  \le \sigma(L) \le \frac{\sym(L)}{1-\sym(L)}.
\]
If $\|\cdot\| = \|\cdot\|_e^*$   for some $e \in \interior(K^*)$ then
\[
\frac{\sym(L)}{1+\sym(L)} = \sigma(L).
\]
\end{theorem}
\begin{proof}
To ease notation, let $s:= \sym(L)$ and $\sigma := \sigma(L)$.  First we show that $\sigma\ge \frac{s}{1+s}$.  To that end, assume $v\in K, \|v\|=1$ is given.   It follows from the definition of $\sym(L)$ that there exists $z\in K, \|z\|\le 1$ such that $z+sv \in L$.  Thus $x:=\frac{1}{\|z+sv\|}(z+sv)\in L, \|x\|=1$ and
\[
\lambda_{v}(x) \ge \frac{s}{\|z+sv\|} \ge \frac{s}{1+s}.
\]  
Since this holds for any $v\in K, \|v\|=1$, it follows that
$\sigma\ge \frac{s}{1+s}.$

Next we show that $\sigma\le\frac{s}{1-s}$.  To that end, assume $v\in K, \|v\| = 1$ is such that 
\begin{equation}\label{eq.sym.less.one}
\max_{x\in K,\,t\in \R}\{t: x+tv \in L, \; \|x\| \le 1\} < 1.
\end{equation}
At least one such $v$ exists because $\sym(L) < 1$.

It follows from the construction of $\sigma(L)$ that there exists $x \in L, \|x\| = 1$ such that $\lambda_v(x) = \sigma > 0$.  In particular, $x - \sigma v \in K$.  Furthermore, $x-\sigma v \ne 0$ as otherwise $v =\frac{1}{\alpha} x\in L$ and $x+v \in L$ which would contradict \eqref{eq.sym.less.one}.  Thus $z:=\frac{x-\sigma v}{\|x-\sigma v\|} \in K, \|z\|=1$ and 
$z +\frac{\sigma}{\|x-\sigma v\|} v \in L$ with
$\frac{\sigma}{\|x-\sigma v\|} \ge \frac{\sigma}{1+\sigma}.$  Since this holds for any $v\in K, \|v\|=  1$ satisfying ~\eqref{eq.sym.less.one}, it follows that $s \ge \frac{\sigma}{1+\sigma}$ or equivalently $\sigma \le \frac{s}{1-s}$.

In the special case when $\|\cdot\| = \|\cdot\|_e^*$ for some $e \in \interior(K^*)$ we have $\|z\| = \ip{e}{z}$ for all $z\in K$.  In particular,
$\|x-\sigma v \| = \ip{e}{x-\sigma v} = \ip{e}{x}-\ip{e}{\sigma v} = \|x\| - \sigma \|v\| = 1-\sigma$ in the previous paragraph and so the second inequality can be sharpened to $s \ge \frac{\sigma}{1-\sigma}$ or equivalently $\sigma \le \frac{s}{1+s}$.
\end{proof}

We also have the following relationship between the Grassmann distance to ill-posedness and the symmetry measure.  

\begin{corollary} 
\label{corol.symmetry} Assume $L\in \Gr(V,m)$ is such that $L \cap \interior(K) \neq \emptyset$ and $\|\cdot \| = \vertiii{\cdot}$.  Then
\[
\min_{u\in K^*\atop \|u\|^*=1} \max_{v\in K \atop \|v\|=1} \ip{u}{v}\cdot \frac{ \sym(L)}{1+\sym(L)} \le \dist(L,\Sigma_m) \le \frac{\sym(L)}{1-\sym(L)}.
\]
In particular, if $\|\cdot \| = \vertiii{\cdot} = \|\cdot \|_2$ then 
\[
\cos(\Theta(K^*,K))\cdot \frac{ \sym(L)}{1+\sym(L)} \le \dist(L,\Sigma_m) \le \frac{\sym(L)}{1-\sym(L)}.
\]
\end{corollary}
\begin{proof} This is an immediate consequence of Theorem~\ref{thm:Grassmann.nu}, Theorem~\ref{thm.symmetry}, and Corollary~\ref{corol.nu.sigma}.
\end{proof}

\section{Symmetric and polyhedral cones}
\label{sec.symm.poly}
In this section we specialize our developments to two important cases. The first one is when $K$ is a symmetric cone and the space $V$ is endowed with an Euclidean Jordan algebra structure.  The second one is the case $V=\R^n,\, K = \R^n_+$.

\subsection{Symmetric cones}
\label{sec.symm.cones}

We next recall some basic material on Euclidean Jordan algebras.  The articles~\cite{SchmA01,SchmA03} and the textbooks~\cite{Baes09,FaraK94} provide a more detailed discussion on this topic and its relevance to optimization. 

Assume $V$ is endowed with a bilinear operation $\circ : V\times V \rightarrow V$ and $e\in V$ is a particular element of $V$.  The triple $(V,\circ,e)$ is an {\em Euclidean Jordan algebra with identity element} if  the following conditions hold:
\begin{itemize}
\item $x\circ y = y \circ x$ for all $x,y\in V$
\item $x\circ(x^2 \circ y) = x^2\circ(x\circ y)$ for all $x,y\in V$, where $x^2 = x\circ x$
\item $x\circ e = x$ for all $x \in V$
\item There exists an associative positive definite bilinear form on $V$.
\end{itemize}
An element $c\in V$ is {\em idempotent} if $c^2  = c$. An idempotent element of $V$ is a {\em primitive idempotent} if it is not the sum of two other idempotents. We will write $c\in \PI$ as shorthand for the statement  `$c$ is a primitive idempotent.' The rank $r$ of $V$ is the smallest integer such that for all $x\in V$ the set $\{e,x,x^2,\dots,x^r\}$ is linearly dependent.  Every element $x\in V$ has a {\em spectral decomposition}
\begin{equation}\label{spectral}
x = \sum_{i=1}^r \lambda_i(x) c_i,
\end{equation}
where $\lambda_i(x) \in \R,\; i=1,\dots,r$ are the {\em eigenvalues} of $x$ and  $\{c_1,\dots,c_r\}$ is a {\em Jordan frame,} that is, a collection of non-zero primitive idempotents such that $c_i \circ c_j = 0$ for $i\ne j$, and $c_1+\cdots+c_r = e$.  The {\em trace} of $x \in V$ is defined as
$$
\tr(x) = \sum_{i=1}^r\lambda_i(x) 
$$
Throughout this section we  assume that $(V,\circ,e)$ is an Euclidean Jordan algebra with identity and let $r$ denote the rank of $V$.  Furthermore, we assume that $V$ is endowed with the following trace inner product:
\begin{equation}\label{inner.Jordan}
\ip{x}{y}:= \tr(x\circ y).
\end{equation}
We also assume that $K$ is the cone of squares in $V,$ that is, $K = \{x^2: x\in V\}.$ It is known that $K$ is a {\em symmetric cone,} that is, $K=K^*$ and for all $u,v\in \text{int}(K)$ there exists a linear automorphism $A\in GL(V)$ such that $Au = v$ and $AK=K.$

The following four cases summarize the main examples of Euclidean Jordan algebras that are popular in optimization.  First, $\R^n$ with the componentwise product $(x \circ z)_i = x_i z_i, \; i=1,\dots,n.$  In this case $K = \R^n_+$ and $r=n$.   Second, $\S^n$ with the product $ X\circ Z = \frac{XZ+ZX}{2}$. In this case $K= \S^n_+$ and $r=n$.  Third, $\R^{n}$ with the product $\matr{x_0 \\ \bar x} \circ \matr{z_0 \\ \bar z} = \matr{x\transp z \\ x_0 \bar z + z_0 \bar x}.$  In this case $K = \{(x_0,\bar x)\in \R^n: x_0 \ge \|\bar x\|_2\}$ and $r=2.$  Fourth, any direct product of the above  types of Euclidean Jordan algebras is also an Euclidean Jordan algebra.

For $x\in V$, the vector of eigenvalues $\lambda(x):=(\lambda_1(x),\dots,\lambda_r(x))$ is nicely related to the  norms $\|x\|_2, \|x\|_e, \|x\|_e^*$ as follows
\begin{equation}\label{eq.various.norms}
\|x\|_2 = \|\lambda(x)\|_2, \; \|x\|_{e} = \|\lambda(x)\|_\infty, \; \|x\|_{e}^* =  \|\lambda(x)\|_1.
\end{equation}
It is also easy to see that the eigenvalue mapping $\lambda_e$ satisfies $\lambda_e(x) = \min_{i=1,\dots,r} \lambda_i(x)$ and the set $\Delta(K)$ can be written in any of the following ways
\[
\{x\in K: \|x\|_e^* = 1\} = \{x\in K: \ip{e}{x} = 1\}= \left\{x\in K: \sum_{i=1}^r \lambda_i(x) = 1\right\}.
\]
Part (a) in Proposition~\ref{prop.nu.sigma.symm} below can be seen as an  extension of the identities \eqref{eqn.nu.lp} and \eqref{eqn.nu.sdp} in Section~\ref{sec.norms}.
\begin{proposition}\label{prop.nu.sigma.symm} Assume $L \in \Gr(V,m)$ is such that $L \cap \interior(K) \ne \emptyset.$   
\begin{description}
\item[(a)] If $\vertiii{\cdot} = \|\cdot\|_e$ then 
\[
\sigma(L) =\nu(L) = \max_{x\in L\cap K \atop \|x\|= 1}\min_{i\in 1,\dots,r} \lambda_i(x). \]
In particular, if $\|\cdot\| = \|\cdot\|_e^*$ then
\[
\sigma(L) = \nu(L) = \max_{x\in L\cap \Delta(K)}\min_{i\in 1,\dots,r} \lambda_i(x).
\]
\item[(b)] 
If $\vertiii{\cdot} = \|\cdot\|_e^*$ then 
\[
\sigma(L) = \nu(L) = \min_{c \in \PI} \max_{x\in L\cap K \atop \|x\|= 1} \ip{c}{x}.
\]
In particular, if $\|\cdot\| = \|\cdot\|_e^*$ then
\[
\sigma(L) = \nu(L) = \min_{c \in \PI}\max_{x\in L\cap \Delta(K)} \ip{c}{x}.\]
(Recall that $c\in \PI$ is shorthand for the statement `$c$ is a primitive idempotent.')

\item[(c)] If $\|\cdot\| = \vertiii{\cdot} = \|\cdot\|_2$ then
\[
\sigma(L) = \nu(L) = \min_{u \in K \atop \|u\|_2 =1} \max_{x\in L\cap K \atop \|x\|_2= 1} \ip{u}{x}.
\]
\end{description}
\end{proposition}
\begin{proof}
\begin{description}
\item[(a)] Since $L \cap \interior(K) \ne \emptyset$ and $\vertiii{\cdot} = \|\cdot\|_e$,  Proposition~\ref{prop.eigenvalue} and Corollary~\ref{corol.nu.sigma}(c) yield
\[
\sigma(L) = \nu(L)=\max_{x\in L\atop \|x\| \le 1}  \lambda_{e}(x) = \max_{x\in L\cap K \atop \|x\| = 1} \min_{i=1,\dots,r} \lambda_{i}(x). 
\]
If in addition $\|\cdot\| = \|\cdot\|_e^*$ then $\|x\| = \ip{e}{x}$ for all $x\in K$.  In other words, $\{x\in K: \|x\| = 1\} = \Delta(K)$.  Thus
\[
\sigma(L) =  \nu(L) = \max_{x\in L\cap K \atop \|x\| = 1} \min_{i=1,\dots,r} \lambda_{i}(x) = \max_{x\in L\cap \Delta(K)} \min_{i=1,\dots,r} \lambda_{i}(x).
\] 
\item[(b)] The first statement is a consequence of the following inequalities, which by putting them together must hold with equality:
\begin{equation}\label{eq.claim.1}
\sigma(L) \le \min_{c \in \PI} \max_{x\in L\cap K \atop \|x\|= 1} \lambda_{c}(x)\le \min_{c \in \PI} \max_{x\in L\cap K \atop \|x\|= 1} \ip{c}{x},
\end{equation}
\begin{equation}\label{eq.claim.2}
\nu(L) \ge \min_{u \in K\atop \|u\|_e=1} \max_{x\in L\cap K \atop \|x\|= 1} \ip{u}{x}= \min_{c \in \PI} \max_{x\in L\cap K \atop \|x\|= 1} \ip{c}{x},
\end{equation}
\begin{equation}\label{eq.claim.3}
\sigma(L)\ge \nu(L).
\end{equation}
To prove \eqref{eq.claim.1}, observe that if  $c\in \PI$  then $\ip{c}{c} = 1$ and consequently $\lambda_c(x) \le \ip{c}{x}$ for all $x\in L.$  Furthermore, if $c\in \PI$  then $\|c\|_e^* = \ip{e}{c} = 1$.  Therefore from \eqref{def.sigma.feas} we get
\[
\sigma(L)=\displaystyle\min_{v\in K\atop \ip{e}{v}=1}  \max_{x\in L\cap K \atop \|x\| = 1}\lambda_v(x) \le  \min_{c \in \PI} \max_{x\in L\cap K \atop \|x\|= 1} \lambda_{c}(x) \le \min_{c \in \PI} \max_{x\in L\cap K \atop \|x\|= 1} \ip{c}{x}.
\]
To prove \eqref{eq.claim.2}, observe that if $u\in K$ and $\|u\|_e =1$ then, from the spectral decomposition of $u$ and \eqref{eq.various.norms}, it follows that $u-c \in K$ for some $c\in \PI$ and so $\ip{u}{x} \ge \ip{c}{x}$ for all $x\in L\cap K.$  Thus Proposition~\ref{prop.nu} implies
\[
\nu(L) = \min_{u\in K \atop \|u\|_e  =1} \max_{x\in L \atop \|x\|=1} \ip{u}{x} 
\ge \min_{u\in K \atop \|u\|_e  =1} \max_{x\in L\cap K \atop \|x\|=1} \ip{u}{x} 
=  \min_{c\in \PI} \max_{x\in L\cap K \atop \|x\|=1} \ip{c}{x}. 
\]
Finally \eqref{eq.claim.3} follows from Corollary~\ref{corol.nu.sigma}(a) and H\"older's inequality~\eqref{eq.holder}.

The second statement follows as in part (a).
\item[(c)] The reasoning is almost identical to that in part (b). This is a consequence of the following inequalities whose proofs are straightforward modifications of the proofs of \eqref{eq.claim.1}, \eqref{eq.claim.2}, and \eqref{eq.claim.3}:
\begin{equation}\label{eq.claim.1.c}
\sigma(L) = \min_{u \in K\atop \|u\|_2=1} \max_{x\in L\cap K \atop \|x\|_2= 1} \lambda_u(x)\le \min_{u \in K\atop \|u\|_2=1} \max_{x\in L\cap K \atop \|x\|_2= 1} \ip{u}{x},
\end{equation}
\begin{equation}\label{eq.claim.2.c}
\nu(L) = \min_{u \in K\atop \|u\|_2=1}  \max_{x\in L \atop \|x\|_2= 1} \ip{u}{x}\ge  \min_{u \in K\atop \|u\|_2=1}  \max_{x\in L\cap K \atop \|x\|_2= 1} \ip{u}{x},
\end{equation}
\begin{equation}\label{eq.claim.3.c}
\sigma(L)\ge \nu(L).
\end{equation}
\end{description}
\end{proof}

Corollary~\ref{corol.precond} below is an immediate consequence of Proposition~\ref{prop.nu.sigma.symm}. Corollary~\ref{corol.precond}(b) shows that if the problem \[
\text{ find} \; w \; \text{ such that }\; Aw \in \interior(K)\]
is feasible, then it can be recast as the {\em preconditioned} equivalent problem 
\[
\text{ find} \; w \; \text{ such that }\; (PAR)w \in \interior(K)\]
for suitable linear automorphisms $P,R$ where the latter reformulation is well-conditioned.

\begin{corollary}\label{corol.precond}  Assume $\|\cdot\| = \vertiii{\cdot} = \|\cdot\|_2$.
\begin{description}
\item[(a)] If $L\in \Gr(V,m)$ is such that $L \cap \interior(K)\ne \emptyset$ then there exists a linear automorphism $P\in GL(V)$ such that $PK = K$ and 
\[ \nu(PL) \ge \frac{1}{\sqrt{r}}.\]
\item[(b)] 
Assume $A:W \rightarrow V$ is injective and 
$\Image(A) \cap \interior(K) \ne \emptyset$.  Then there exist linear automorphisms $P\in GL(V), R\in GL(W)$ such that $PK = K$ and
\[
\frac{\|PAR\|}{\Rdist(PAR,\Sigma_m)} \le \sqrt{r}.
\]

\end{description}
\end{corollary}
\begin{proof}
\begin{description}
\item[(a)] 
Let $x_0 \in L \cap \interior(K)$.  Since $K$ is symmetric, there exists $P\in   GL(V)$ such that $PK=K$ and $Px_0 = e.$ Observe that $e\in PL\cap \interior(K)$, $\|e\|_2 = \sqrt{r},$ and $\ip{e}{u} = \|u\|_e \ge \|u\|_2$ for all $u\in K$.  Thus Proposition~\ref{prop.nu.sigma.symm}(c) implies that $\nu(PL) \ge \frac{1}{\sqrt{r}}$.
\item[(b)] Let $L = \Image(A)$ and $P$ be as in part (a).  Let $R\in GL(W)$ be such that $PAR$ is balanced, that is, $(PAR)^*(PAR) = I.$   Thus, by Theorem~\ref{thm.renegar},
\[
\frac{\|PAR\|}{\Rdist(PAR,\Sigma_m)} = \frac{1}{\nu(PL)} \le \sqrt{r}.
\]
\end{description}
\end{proof}

\subsection{Polyhedral case}
We now consider the special case $V = \R^n, K = \R^n_+$.  The peculiar structure of $\R^n_+$ enables a refinement of the the former condition measures  to all $L \in \Gr(\R^n,m)$ including $\Sigma_m$ via a suitable canonical partition of $\{1,\dots,n\}$.   More precisely,  the classical Goldman-Tucker partition theorem implies that for all $L \in \Gr(\R^n,m)$ there exists a unique partition $B\cup N = \{1,\dots,n\}$ such that 
\begin{equation}\label{eq.partition}
L \cap (\interior{\R^B_+}\times\{0_N\}) \ne \emptyset \; \text{ and } \; L^\perp \cap (\{0_B\} \times \interior{\R^N_+}) \ne \emptyset.
\end{equation}
(A detailed proof of the existence of this partition can be found in~\cite{CheuCP03}.)

Note that $L\not \in \Sigma$ if and only if either $B=\emptyset$ or $N=\emptyset$.  Regardless of whether $L \in \Sigma_m$, the partition $B,N$  yields a decomposition into two subproblems each of which is well-posed as we detail next.

Let $V_B:= \R^B\times \{0\}, K_B:= \R^B_+\times \{0\}, \; V_N:=  \{0\}\times\R^N, K_N := \{0\}\times  \R^N_+$ and $L_B:=L\cap V_B, L_N = L^\perp\cap V_N$. Observe that relative to the space-cone pairs $(V_B,K_B)$ and $(V_N,K_N)$ both $L_B$ and $L_N$ are respectively well-posed since $L_B \cap \interior(K_B) \ne\emptyset$ and $L_N \cap \interior(K_N) \ne\emptyset$.  Therefore, the partition $(B,N)$ yields a natural decomposition of the entire previous machinery.  More precisely, we have 
\begin{equation}\label{eq.nu}
\nu(L_B) = \min_{u\in K_B, y\in L_B^\perp\atop \vertiii{u}^*=1} \|u-y\|^*,
\quad \nu(L_N) = \min_{u\in K_N, y\in L_N^\perp\atop \vertiii{u}^*=1} \|u-y\|^*,
\end{equation} 
and 
\begin{equation}\label{eq.sigma}
\sigma(L_B) = 
\min_{v\in K_B\atop \vertiii{v}=1} \max_{x\in L_B\atop \|x\|\le 1}\lambda_v(x), \quad
\sigma(L_N) = 
\min_{v\in K_N\atop \vertiii{v}=1} \max_{x\in L_N\atop \|x\|\le 1}\lambda_v(x).
\end{equation} 
Furthermore, if $k = \dim(L_B)$ and $\ell = \dim(L_N)$ then 
\[
\dist(L_B,\Sigma_{k,B}) = \nu(L_B), \quad \dist(L_N,\Sigma_{\ell,N}) = \nu(L_N)
\]
where $\Sigma_{k,B} \subseteq \Gr(V_B,k)$ and $\Sigma_{\ell,N} \subseteq \Gr(V_N,\ell)$ are respectively the set of $k$-dimensional and $\ell$-dimensional ill-posed subspaces in $V_B$ and $V_N$ respectively.

The following result readily follows from Proposition~\ref{prop.nu.sigma.symm} specialized to each of the two subproblems defined by the partition $(B,N)$.   Proposition~\ref{prop.ye.cond}(b) below shows that $\sigma(L_B)$ and  $\sigma(L_N)$ coincide with the sigma measure introduced by Ye~\cite{Ye94}.

\begin{proposition}\label{prop.ye.cond} Assume $L \in \Gr(\R^n,m)$ and let $(B,N)$ is the partition determined by  $L$ as in \eqref{eq.partition}.  
\begin{description}
\item[(a)] If $\vertiii{\cdot} = \|\cdot\|_\infty$ then 
\[
\sigma(L_B) =\nu(L_B) = \max_{x\in L\cap \R^n_+ \atop \|x\|= 1}\min_{i\in B} x_i, \quad
\sigma(L_N) =\nu(L_N) = \max_{x\in L^\perp\cap \R^n_+ \atop \|x\|= 1}\min_{i\in N} x_i
\]
In particular, if $\|\cdot\| = \|\cdot\|_1$ then
\[
\sigma(L_B) = \nu(L_B) = \max_{x\in L\cap \Delta_{n-1}}\min_{i\in B} x_i, \quad
\sigma(L_N) = \nu(L_N) = \max_{x\in L^\perp\cap \Delta_{n-1}}\min_{i\in N} x_i.
\]
\item[(b)] 
If $\vertiii{\cdot} = \|\cdot\|_1$ then
\[
\sigma(L_B) =\nu(L_B) = \min_{i\in B} \max_{x\in L\cap \R^n_+ \atop \|x\|= 1} x_i, \quad
\sigma(L_N) =\nu(L_N) = \min_{i\in N}\max_{x\in L^\perp\cap \R^n_+ \atop \|x\|= 1} x_i
\]
In particular, if $\|\cdot\| = \|\cdot\|_1$ then
\[
\sigma(L_B) = \nu(L_B) = \min_{i\in B}\max_{x\in L\cap \Delta_{n-1}} x_i, \quad
\sigma(L_N) = \nu(L_N) = \min_{i\in N}\max_{x\in L^\perp\cap \Delta_{n-1}} x_i.
\]
\item[(c)] If $\|\cdot\| = \vertiii{\cdot} = \|\cdot\|_2$ then
\[
\sigma(L_B) =\nu(L_B) = \min_{u\in \R^B_+ \times \{0_N\}\atop \|u\|_2=1} \max_{x\in L\cap \R^n_+ \atop \|x\|_2= 1} \ip{u}{x}, \]
and
\[
\sigma(L_N) =\nu(L_N) = \min_{u\in \{0_B\}\times \R^N_+ \atop \|u\|_2=1}\max_{x\in L^\perp\cap \R^n_+ \atop \|x\|_2= 1} \ip{u}{x}.
\]

\end{description}
\end{proposition}

We conclude with a discussion analogous to that in Section~\ref{sec.renegar}.  Assume $A:\R^m \rightarrow \R^n$ is an injective linear mapping and $L = \Image(A)$.  The partition  $(B,N)$ determined by $L$ in turn yields the following decomposition of $A$ and $\R^m$.  First, the mapping $A$ decomposes as $A = \matr{A_B \\ A_N}$.  Second, the space $\R^m$ has the orthogonal decomposition $\R^m = \R^m_B + \R^m_N$ where $\R^m_B:=\Image(A_B\transp)$ and $\R^m_N:=\ker(A_B) = (\R^m_B)^\perp.$  Consequently, the mapping $A$ can be written in the following block matrix form
\[
A = \matr{A_B \\ A_N} = \matr{A_{BB} & 0 \\  A_{NB} & A_{NN}},
\]
where $A_{BB}$ and $A_{NN}$ are both full column-rank
and the spaces $L_B$ and $L_N$ are
\[
L_B = \Image(A_{BB}) \times \{0_N\}, \quad L_N = \{0_B\}\times\ker(A_{NN}\transp).
\]  
Theorem~\ref{thm.renegar} readily yields 
\[
\frac{1}{\|A_{BB}\|} \le \frac{\nu(L_B)}{\Rdist(A_{BB},\Sigma)} \le \|A_{BB}^{-1}\|,
\]
and
\[
\frac{1}{\|A_{NN}\|} \le \frac{\bar\nu(L_N^\perp)}{\Rdist(A_{NN},\Sigma)} \le \|A_{NN}^{-1}\|.
\]
Furthermore, proceeding as in Corollary~\ref{corol.precond}, it follows that if $\|\cdot\| = \vertiii{\cdot} = \|\cdot\|_2,$ then for a suitable positive diagonal matrix $D$ and non-singular matrix $R$ the  space $DL = \Image(\hat A)$ of the preconditioned matrix $\hat A = DAR$ determines the same partition $(B,N)$ as $L=\Image(A)$ and the preconditioned matrix $\hat A$ satisfies
\[
\frac{\|\hat A_{BB}\|}{\Rdist(\hat A_{BB},\Sigma)} = \frac{1}{\nu((DL)_B)} \le \sqrt{k},\]
and
\[ \frac{\|\hat A_{NN}\|}{\Rdist(\hat A_{NN},\Sigma)} = \frac{1}{\nu((DL)_N)} \le \sqrt{\ell}.
\]


\end{document}